\documentclass[pdflatex,sn-mathphys-num]{sn-jnl}

\usepackage{graphicx}%
\usepackage{multirow}%
\usepackage{amsmath,amssymb,amsfonts}%
\usepackage{amsthm}%
\usepackage{mathrsfs}%
\usepackage[title]{appendix}%
\usepackage{xcolor}%
\usepackage{textcomp}%
\usepackage{manyfoot}%
\usepackage{booktabs}%
\usepackage{algorithm}%
\usepackage{algorithmicx}%
\usepackage{algpseudocode}%
\usepackage{listings}%

\theoremstyle{thmstyleone}%
\newtheorem{thm}{Theorem}
\newtheorem{cor}[thm]{Corollary}

\newtheorem{lem}[thm]{Lemma}

\theoremstyle{thmstyletwo}%
\newtheorem{exa}{Example}%
\newtheorem{rem}{Remark}%

\theoremstyle{thmstylethree}%
\newtheorem{defi}{Definition}%
\newtheorem*{defibis}{Definition 3.bis}

\begin{document}

\title[Article Title]{On linearly ordered sets of chain components}

\author[1]{\fnm{Patrizio} \sur{Cintioli}}\email{patrizio.cintioli@unicam.it}
\equalcont{These authors contributed equally to this work.}

\author*[1]{\fnm{Alessandro} \sur{Della Corte}}\email{alessandro.dellacorte@unicam.it}
\equalcont{These authors contributed equally to this work.}

\author[2]{\fnm{Marco} \sur{Farotti}}\email{marco.farotti@unicam.it}
\equalcont{These authors contributed equally to this work.}

\affil[1]{\orgdiv{Mathematics Division}, \orgname{University of Camerino}, \orgaddress{\street{via Madonna delle Carceri 9}, \city{Camerino (MC)}, \postcode{62032}, \country{Italy}}}

\affil[2]{\orgdiv{Doctoral School in Computer Sciences and Mathematics}, \orgname{University of Camerino}, \orgaddress{\street{via Madonna delle Carceri 9}, \city{Camerino (MC)}, \postcode{62032}, \country{Italy}}
}

\abstract{In a dynamical system $(X,f)$, with $X$ a compact metric space, the chain components, the fundamental building blocks in the Conley decomposition of dynamics, have a natural partial order induced by the chain relation between points. Although chain components are crucial for understanding the long-term behavior of topological systems, they have not been widely studied from the point of view of poset theory. 

In this work, we pursue this line of research, considering both the case in which $f$ is a continuous map and the general case in which no regularity assumption is made. 
Our main results are that, if $f$ is continuous:

$\bullet$ the chain components poset cannot be a linearly and densely ordered set;
    
$\bullet$ every countable well-order with a maximum is the order type of the chain components poset of an interval map.

 If no regularity assumption is made:
 
$\bullet$ there is a dynamical system on the interval whose chain components poset is countable and densely ordered;
    
$\bullet$ the chain components poset has at least
one minimal element.

These results, bridging dynamical systems and order theory, highlight both the structural constraints and the possibilities for the chain components posets.
}

\keywords{Topological dynamics; Chain recurrence; Discontinuous dynamical systems; Dense orders.}

\pacs[MSC Classification]{37B20, 37B30, 06A06}

\maketitle

\section{Introduction}\label{sec:1}
The structure of the set of chain recurrent points in a dynamical system has long been a subject of interest in topological dynamics, particularly through the lens of Conley's theory. 
Among the tools used to analyze topological dynamical systems, the partial ordering of chain components, induced by the chain relation, offers a compact and effective way to encode dynamical complexity.

\noindent In this paper, we study the chain components poset (in short, the $\mathcal{C}$-poset) associated with a dynamical system, with particular focus on the realizability of certain order types under some fairly natural assumptions on the system. 
We aim to clarify the interplay between the regularity of the map, the topological nature of the phase space, and the possible ordering structures that emerge from the dynamics.

\noindent Let $(X,d)$ be a metric space, and $f:X\to X$ a map. 
We say that the pair $(X,f)$ is a \emph{topological dynamical system} (or simply a dynamical system). Notice that, in general, we are not assuming any regularity property for $f$. We will say that $(X,f)$ is a \emph{continuous dynamical system} if $f$ is assumed to be a continuous map.

\noindent We will indicate by $\mathbb{N}$ 
and $\mathbb{N}_0$, respectively, the set of positive and non-negative integers.

\begin{defi}\label{def_inv}
    Let $(X,f)$ be a dynamical system, and consider $S\subseteq X$. We say that the set $S$ is \emph{$f$-invariant} (or simply invariant) if $f(S)\subseteq S$.
\end{defi}
\begin{defi}\label{def_subsy}
    Let $(X,f)$ be a dynamical system. For $Y\subseteq X$, we say that $(Y,f\vert_Y)$ is a \emph{subsystem} of $(X,f)$ if $Y$ is closed and invariant. 
\end{defi}

\noindent Let us recall the notion of $\varepsilon$-{\em chain} (or $\varepsilon$-{\em pseudo-orbit}) developed in \cite{co78}.
\begin{defi}\label{def_chain_rel}
Let $(X,f)$ be a dynamical system. Given two points $x,y\in X$ and $\varepsilon>0$, an $\varepsilon$-\emph{chain} from $x$ to $y$ is a finite set of points $x_0,x_1,\ldots,x_n$ in $X$, with $n\ge 1$, such that
\begin{itemize}
 \item[(i)] $x_0=x$ and $x_n=y$,
 \item[(ii)] $d(f(x_i),x_{i+1})<\varepsilon$ for every $i=0,1,\ldots, n-1$.
\end{itemize}
The {\it chain relation} is the binary relation $\mathcal{C}\subseteq X^2$ defined as follows: given $x,y\in X$, 
\[
x\,\mathcal{C}\, y\iff \text{for every $\varepsilon>0$ there exists an $\varepsilon$-chain from $x$ to $y$}.
\]
\end{defi}

\noindent A point $x\in X$ is called {\it chain recurrent} if $x\,\mathcal{C}\,x$. 
The set of all the chain recurrent points of a dynamical system $(X,f)$ is denoted by $CR(X,f)$.
We write simply $CR_f$ when the space $X$ is clear from the context.
Given a dynamical system $(X,f)$, we define the equivalence relation $E$ on $CR(X,f)$ as follows: for every $x,y\in CR(X,f)$,
\[
x\,E\,y\iff x\,\mathcal{C}\, y\text{ and }y\,\mathcal{C}\, x.
\]
Let $\mathfrak{C}(X,f):=CR(X,f)/E$ be the set of equivalence classes on $CR(X,f)$ with respect to the relation $E$. 
We write simply $\mathfrak{C}_f$ when the space $X$ is clear from the context.
The equivalence classes of $\mathfrak{C}_f$ are usually called \emph{chain components} (\cite[Def.~2.70]{kurka}). 
We define the partial order relation $\preceq$ on $\mathfrak{C}_f$ in the usual way through representatives. 
Before introducing this partial order on chain components, we briefly comment on the direction chosen for the relation. Of course both directions are a priori admissible in defining the relation, but we adopt the convention, common in the literature, that $[x] \preceq [y]$ whenever $y \,\mathcal{C}\, x$, that is, when the dynamics flows from $y$ to $x$ through pseudo-orbits with arbitrarily small corrections. 
This choice aligns naturally with the role of Lyapunov functions, which decrease along orbits and are constant precisely on the chain components. 
Intuitively, one can think of $[y]$ ``collapsing'' dynamically into $[x]$, as energy dissipates or as the system evolves towards a more stable behavior. 
In this sense, the order translates into a form of dynamical causality if we think that ``larger" components flowing towards ``smaller" ones. 

\noindent Given $[x],[y]\in \mathfrak{C}_f$, we set:
\[
[x]\preceq [y]\iff y\,\mathcal{C}\,x.
\]
The poset $(\mathfrak{C}_f,\preceq)$ is called \emph{chain components poset} (in short, the $\mathcal{C}$-poset) of the system $(X,f)$.

\noindent We can also give the following purely topological definition of the chain relation $\mathcal{C}$.

\begin{defibis}
Let \((X, f)\) be a dynamical system, and let \(x, y \in X\). The chain relation $\mathcal{C}$ on $X$ is the binary relation defined as follows: given $x, y \in X$, $x\,\mathcal{C}\, y$ if and only if for every neighborhood $U\subseteq X\times X$ of the diagonal there exists a finite sequence of points \(x = x_0, x_1, \ldots, x_n = y\), called a chain, such that for each \(i = 0, 1, \ldots, n-1\),  
\[ 
(f(x_i), x_{i+1}) \in U. 
\]
\end{defibis}

\begin{rem}
The two definitions \ref{def_chain_rel} and 3.\textit{bis}. are equivalent (see for instance \cite[p.~47]{kurka}).
\end{rem}

\begin{defi}\label{def_dens_ord_}
    We say that a poset $(A,\le)$ is \emph{densely ordered} on some subset $S\subseteq A$ (or simply that the subset $S$ is densely ordered) if, for every $x,y\in S$ such that $x<y$, there exists $z\in S$ such that $x<z<y$.
\end{defi}
\begin{rem}
    Notice that if a poset $(A,\le)$ has a densely ordered subset $S\subseteq A$, then $|S|\ge \aleph_0$.
\end{rem}

\begin{defi}\label{def_scattered}
A poset $(A,\leq)$ is said to be \emph{scattered} if it does not contain a subset order-isomorphic to the rationals $(\mathbb{Q}, \leq)$.  
\end{defi}

\begin{defi}\label{def_directedset}

    A \emph{directed set} is a pair $(\Lambda,\leq)$ consisting of a nonempty set
$\Lambda$ and a reflexive and transitive relation $\leq$ on $\Lambda$ such that
for every $\lambda_1,\lambda_2 \in \Lambda$ there exists $\lambda_3 \in \Lambda$
with
$$\lambda_1 \leq \lambda_3 \quad \text{and} \quad \lambda_2 \leq \lambda_3.$$
    
\end{defi}

\begin{defi}\label{def_netcluster}
    Let $(\Lambda,\leq)$ be a directed set and $X$ a topological space.
    A \emph{net} in $X$ is a function $x : \Lambda \to X$. We will denote the net by
    $\{x_\lambda\}_{\lambda \in \Lambda}$.
    For a net $\{x_\lambda\}_{\lambda \in \Lambda}$ in $X$, a point $z \in X$ is called a \emph{cluster point}
    of the net if for every neighbourhood $U$ of $z$ and every $\lambda_0 \in \Lambda$
    there exists $\lambda \in \Lambda$ with $\lambda \geq \lambda_0$ such that 
    $x_\lambda \in U.$
\end{defi}
\noindent Let us recall the following basic result on topological quotients (cf.~Proposition 5.9 in \cite[p.~94]{ma15}):
\begin{thm}\label{prop_quot}
    Let $X,Y$ be topological spaces and $f: X\rightarrow Y$ continuous function. Let $\sim$ be an equivalence relation on $X$ and $\pi : X\rightarrow X/\sim$ the canonical quotient map. 
    There exists a continuous map $$g: X/\sim\, \longrightarrow \, Y$$ such that $g \circ \pi =f$ if and only if $f$ is constant on equivalence classes.
\end{thm}
\begin{defi}
    We say that a subset $S\subseteq X$ of a topological space is \emph{nowhere dense} if $\left( \overline{S}\right)^\circ=\emptyset$, that is, the interior of its closure is empty or equivalently, if $S$ is not dense in any nonempty open subset $U\subseteq X$.
\end{defi}

\begin{defi}
    We say that a topological space is \emph{totally disconnected} if its only connected subspaces are one-point sets.
\end{defi}

\begin{defi}\label{def_noncompact}
    We say that a topological space $X$ is \emph{second countable} if it has a countable basis, that is, there exists a countable family of open sets $\{U_1,U_2,\ldots\}$ which form a base of the topology of $X$. 
    Moreover, we say that $X$ is \emph{locally compact} if every point $x\in X$ has a compact neighbourhood, that is, there exists an open set $U$ and a compact set $K$, such that $x\in U\subseteq K$.
\end{defi}
\begin{defi}\label{def_com_Lyap}
A \emph{complete Lyapunov function} for $f:X\rightarrow X$ is a continuous map $L: X \rightarrow \mathbb{R}$
with the following properties:
\begin{itemize}
    \item[(i)] $L(f(x)) \le L(x)$ for all $x\in X$ , with equality if and only if $x\in CR_f$.
    \item[(ii)] $L$ is constant on each chain component, that is, for every $[x]\in \mathfrak{C}_f$ and for every $u,v\in [x]$ 
    $$
    L(u)= L(v),
    $$
    and takes different values on different chain components.
    \item[(iii)] If $[x]$ and $[y]$ are distinct chain components such that $x\,\mathcal{C}\, y$, then $L(x) > L(y)$.
    \item[(iv)] $L(CR_f)$ is a compact nowhere dense subset of $\mathbb{R}$.
\end{itemize}
\end{defi}

\noindent We will exploit some facts from Conley’s theory (\cite{co78, mischaikow}). In particular, the following theorem holds (see \cite[Theorem 4.9]{norton}).
\begin{thm}\label{conley_th}
Let $X$ be a compact metric space
and let $f:X\rightarrow X$ be a continuous function. Then there is a complete Lyapunov function for $f$.

\qed
\end{thm}

\noindent A central result of this work is the tight incompatibility between continuity and dense linear ordering: a countable chain components poset that is both linearly and densely ordered cannot arise from a continuous dynamical system. 
Note that the result is sharp in the following sense: at the moment either continuity or density is relaxed, allowing a countable discontinuity set, such posets become realizable. 
In particular, dropping continuity allows for a countable densely ordered poset, while dropping density allows for one under continuity.

\noindent A second result shows a similarly sharp contrast regarding minimal elements: while in non-compact metric spaces it is possible to construct systems whose chain components poset has no minimal elements, compactness forces their existence. In any compact metric space, the presence of at least one minimal element is unavoidable.

\noindent A further result of this work concerns the range of well-orders that can be realized through continuous interval maps. 
We show that every order type of the form $(\lambda + 1)$, with $\lambda < \omega_1$, arises as chain components poset of a suitably constructed continuous dynamical system on the interval \([0,1]\). 
The restriction \(\lambda < \omega_1\) is essential in our construction, as it relies on the existence of a countable cofinal sequence in $\lambda$. 
This result illustrates the richness of continuous dynamics even under strong topological constraints.

\section{Linearly ordered structures in $\mathcal{C}$-posets of continuous systems}
Let us first establish some results for continuous dynamical systems. 
In the following results, we show that: 
\begin{enumerate}
    \item On the one hand, if the space is compact, then the chain recurrent set cannot have a closed subset that generates a subset of the $\mathcal{C}$-poset that is linearly and densely ordered (Theorem~\ref{main}). 
    In particular, the $\mathcal{C}$-poset of the system cannot be linearly and densely ordered, and the analog applies to each of its subsystems (Corollary~\ref{cor_subsystem}). 
    Moreover, Theorem~\ref{main} is also true without assuming compactness, if some conditions are verified (Theorem~\ref{th_cc_loccomp}).
    \item On the other hand, every countable well-order with a maximum is the order type of the $\mathcal{C}$-poset of a continuous interval map (Theorem~\ref{ordinal_order_type}).
\end{enumerate} 

\subsection{Compact spaces}

\begin{thm}\label{main}
Let $X$ be a compact metric space and $f:X\to X$ be a continuous function. Then there is no closed set $J_f\subseteq CR_f$ such that the poset $(J_f/E,\preceq)$ is linearly and densely ordered. 
In particular, the poset $(\mathfrak{C}_f,\preceq)$ cannot be linearly and densely ordered.
\end{thm}
\begin{proof}
    Let $(X,f)$ be a continuous dynamical system with $X$ compact metric space.
    Suppose, towards a contradiction, that there is a closed set $J_f\subseteq CR_f$ such that the poset $(J_f/E,\preceq)$ is linearly and densely ordered. 
    We set $\mathfrak{J}_f=J_f/E$.
    
   \noindent By Theorem~\ref{conley_th}, there exists a complete Lyapunov function for $f$, that is, a continuous function
    \[
    L : X\to \mathbb{R}
    \]
    that verifies the properties (i), (ii), (iii) and (iv) of Def.~\ref{def_com_Lyap}. 
    In particular, property (iv) is satisfied because it is possible to construct the function $L$ so that it verifies the following property (see for instance \cite[Theorem 4.9]{norton}):
    \begin{itemize}
    \item[(iv$'$)] $L(CR_f)$ is a closed subset of the Cantor middle-third set and, in particular it is a compact nowhere dense subset of $\mathbb{R}$.
    \end{itemize}
    So, specifically, $L(CR_f)$ is a totally disconnected, perfect subset of the real numbers. 

    \noindent Then, the restriction $L':J_f\to\mathbb{R}$ of $L$ to $J_f$ is also continuous and by the properties (ii) and (iii) of $L$ we have that $L'$ verifies the following:
    \begin{enumerate}
        \item[(1)] $L'$ is constant on each chain component in $\mathfrak{J}_f$ and takes different values on different chain components.
        \item[(2)] If $[x]$ and $[y]$ are distinct chain components such that $x\,\mathcal{C}\, y$, then $L'(x)>L'(y)$.
    \end{enumerate}
    Since $J_f$ is closed in the compact metric space $X$, it is compact and $L'(J_f)=L(J_f)$ is also compact. 
    Moreover, by (iv$'$) we see that $L'(J_f)=L(J_f)\subseteq L(CR_f)$ is a subset of the Cantor middle-third set. Thus the map $L'$ verifies that:
    \begin{enumerate}
        \item[(3)] $L'(J_f)$ is a closed subset of the Cantor middle-third set and, in particular it is a compact nowhere dense subset of $\mathbb{R}$. 
    \end{enumerate}
    
    \noindent Let $\pi:J_f\to \mathfrak{J}_f$ be the canonical quotient map.
    Since $L'$ satisfies (1), by Theorem~\ref{prop_quot}, there exists a continuous map $\widetilde{L'}:\mathfrak{J}_f\to \mathbb{R}$ such that $\widetilde{L'}\circ\pi=L'$.
    
    \noindent The map $\widetilde{L'}$ is injective. Indeed, let $[x],[y]\in \mathfrak{J}_f$ be two distinct chain components.
    Since $L'$ verifies (1) we have:
    $$
    \widetilde{L'}([x])=\widetilde{L'}(\pi(x))=L'(x)\neq L'(y)=\widetilde{L'}(\pi(y))=\widetilde{L'}([y]).
    $$
    Therefore, $\widetilde{L'}$ is a continuous bijective map between $\mathfrak{J}_f$ and $\widetilde{L'}(\mathfrak{J}_f)$.
    Moreover, $\widetilde{L'}$ is an order-isomorphism between $(\mathfrak{J}_f,\preceq)$ and $(\widetilde{L'}(\mathfrak{J}_f),\le)$. Indeed, for $[x],[y]\in \mathfrak{J}_f$, if $[y]\preceq [x]$, then
$$
\widetilde{L'}([y])=\widetilde{L'}(\pi(y))=L'(y)\le L'(x)=\widetilde{L'}(\pi(x))=\widetilde{L'}([x]),
$$
where the inequality follows from properties (1) and (2) satisfied by $L'$.

\noindent Therefore, since $(\mathfrak{J}_f,\preceq)$ is densely ordered, it follows that $(\widetilde{L'}(\mathfrak{J}_f),\le)$ is also densely ordered. 
By property (3) we have, moreover,  that 
\begin{equation}
\label{totdics}
\widetilde{L'}(\mathfrak{J}_f)=\widetilde{L'}(\pi(J_f))=L'(J_f)\ \text{is a compact nowhere dense subset of $\mathbb{R}$}.
\end{equation}
\noindent Now, we prove that since $\widetilde{L'}(\mathfrak{J}_f)$ is compact and densely ordered in $\mathbb{R}$, then $\widetilde{L'}(\mathfrak{J}_f)$ is an interval in $\mathbb{R}$, which contradicts \eqref{totdics}.

\noindent Suppose, towards a contradiction, that $\widetilde{L'}(\mathfrak{J}_f)$ is not an interval.
Then there is $w\in\mathbb{R}\setminus \widetilde{L'}(\mathfrak{J}_f)$ such that $a<w<b$ for some $a,b\in \widetilde{L'}(\mathfrak{J}_f)$. 
Define 
$$
s:=\sup\{ x\in \widetilde{L'}(\mathfrak{J}_f) : x<w\}
\quad,\quad r:=\inf\{x\in \widetilde{L'}(\mathfrak{J}_f) : x>w\}.
$$
Since $\widetilde{L'}(\mathfrak{J}_f)$ is closed, we have $s,r\in \widetilde{L'}(\mathfrak{J}_f)$. 
Since $\widetilde{L'}(\mathfrak{J}_f)$ is densely ordered, there exists $z\in \widetilde{L'}(\mathfrak{J}_f)$ such that $s<z<r$. 
In particular, $z\neq w$.\\
If $z<w$, then $z\in \{x\in \widetilde{L'}(\mathfrak{J}_f) : x<w\}$, and by definition of $s$ we have $z\le s$, a contradiction. 
Similarly, if $z>w$, then $z\in \{x\in \widetilde{L'}(\mathfrak{J}_f) : x>w\}$, from which follows that $z\ge r$, again a contradiction.
Hence $\widetilde{L'}(\mathfrak{J}_f)$ is an interval in $\mathbb{R}$.
This contradicts the assumption that $\widetilde{L'}(\mathfrak{J}_f)$ is nowhere dense (see Eq.~\eqref{totdics}), and the proof is complete.  
\end{proof}

\noindent The following result is a straightforward consequence of the previous theorem.

\begin{cor}\label{cor_subsystem}
Let $X$ be a compact metric space and $f:X\rightarrow X$ be a continuous map. Then there is no subsystem $(Y,f|_Y)$ such that the poset $(\mathfrak{C}(Y,f),\preceq)$ is linearly and densely ordered.    
\end{cor}

    \noindent Theorem \ref{main} is perhaps a bit subtler than it may look at first sight. Notice, in particular, that it does \emph{not} prevent the existence of a densely and totally ordered subset of the $\mathcal{C}$-poset, and thus it does not prove that the chain components poset is scattered. To show the point, consider the following example. 
    
\begin{exa}\label{cantor__}
    It is possible to have a dynamical system whose chain components poset has a linearly ordered, uncountable (cardinality $2^{\aleph_0}$) subset. 
    Consider, for instance, a continuous function $f:[0,1]\rightarrow [0,1]$ such that $f(x)=x$ for all $x\in C$ where $C$ is the Cantor middle-third set and $f$ is strictly increasing and $f(x)<x$ if $x\in [0,1]\setminus C$. Recalling that $[0,1]\setminus C$ is a countable union of nonempty disjoint open intervals, it is clear that a function $f$ with the above properties can be made continuous. 
    In this example, we have that $CR_f=C$ and that every point in $C$ is a distinct chain component in $\mathfrak{C}_f$, that is, $\mathfrak{C}_f=\{ \{x\} : x\in C\}$. 
    In particular, for $x,y\in C$ we have $[x]\prec [y]$ whenever $x<y$.
\end{exa}

Theorem \ref{main} tells us that: 
    \begin{itemize}
        \item the chain components poset in Example \ref{cantor__} cannot be linearly and densely ordered, and indeed $CR_f=C$ is not densely ordered (in the sense of Def.~\ref{def_dens_ord_}); indeed, if $x$ and $y$ are the two endpoints of any of the open intervals $(x,y)$ removed in the iterative construction of $C$, we have that $(x,y)\cap C=\emptyset$. Thus $\mathfrak{C}_f$ is not densely ordered with respect to $\preceq$;
        \item on the other hand, if we remove from $C$ the set 

    $$\bigcup_{n=1}^{\infty} \left\{ \frac{m}{3^{n+1}} \mid m \in \mathbb{N}, \, m \text{ odd}, \, m < 3^{n+1} \right\},$$
    
we obtain a subset $J\subsetneq CR_f$ which is densely ordered in $\mathbb{R}$ and so $(J/E,\preceq)$ is linearly and densely ordered. 
Notice, however, that $J$ \emph{is not closed}, which is of course consistent with Theorem~\ref{main}.
    \end{itemize}

\noindent After establishing limitations imposed by continuity and compactness, we now turn to a positive result: we show that every ordinal of the form $\lambda+1$, with $\lambda<\omega_1$,  can be realized as the order type of the poset of chain components of a continuous map $f_\lambda:[0,1]\to [0,1]$. This result illustrates that, although dense linear orders are excluded under continuity, there is still ample room for highly structured linear orderings to arise.
\begin{defi}
    Suppose that $S$ is a set and $\{f_\iota:X_\iota\to X_\iota\}_{\iota\in S}$ is a collection of functions. Set $X:=\cup_{\iota\in S}X_\iota$. Assume that, for every $\iota,\xi\in S$, whenever $x\in X_\iota\cap X_\xi$ we have $f_\iota(x)=f_\xi(x)$. 
    
   \noindent  We indicate by $\bigcup_{\iota\in S}f_\iota$ the function $f:X\to X$ defined as $f(x):=f_\iota(x)$ if $x\in X_\iota$.
\end{defi}

\begin{thm}\label{ordinal_order_type} 
For every countable ordinal $ \lambda<\omega_1$, there is a continuous dynamical system $([0,1],f_\lambda)$ such that the poset $(\mathfrak{C}_{f_\lambda},\preceq)$ has order type $(\lambda+1, \in)$.
\end{thm}
\begin{proof}
Let $\lambda<\omega_1$ be a countable ordinal. 
We will construct a countable collection 
    \[
    \Big \{([0,1],f_\alpha)\Big \}_{0\le\alpha\le\lambda}
    \]
    of continuous dynamical systems such that, for every $\alpha\le\lambda$, the $\mathcal{C}$-poset $(\mathfrak{C}_{f_\alpha},\preceq)$ has order type $(\alpha+1, \in)$.
    
    \noindent We start by defining $f_0\equiv id_{[0,1]}$. Therefore, we have $\mathfrak{C}_{f_0}=\{[0]\}$. 
    
    \noindent Let us now proceed by transfinite induction up to $\lambda$. Let us define the map $f_1:[0,1]\rightarrow [0,1]$ as $f_1(x)=x^2$.
    Since $f_1'(0)<1$ and $f_1'(1)>1$, it follows that $0$ and $1$ are, respectively, attractive and repulsive fixed points. Since $f_1$ is strictly increasing and $f(x)<x$ if $x\in (0,1)$, it follows that $\mathfrak{C}_{f_1}=\{[0],[1]\}$ and $[0]\prec [1]$. Therefore, the poset $(\mathfrak{C}_{f_1},\preceq)$ has order type $(2,\in)$. 
    Let us now assume that for every $1\le \alpha<\lambda$ there exists a continuous dynamical system $([0,1],f_\alpha)$ that verifies the following properties: 
 \begin{itemize}
     \item[(i)] $0$ and $1$ are fixed points for the map $f_\alpha$,
     \item[(ii)] $f_\alpha$ is strictly increasing and $f_\alpha(x)<x$ if $x\in [0,1]\setminus CR_{f_\alpha}$,
     \item[(iii)] $(\mathfrak{C}_{f_\alpha},\preceq)$ has order type $(\alpha+1,\in)$,
     \item[(iv)] for $[x],[y]\in \mathfrak{C}_{f_\alpha}$, if $x<y$ then $y\,\mathcal{C}\, x$, that is, $[x]\prec [y]$.
 \end{itemize}
\noindent We want to prove that there is a system $([0,1],f_\lambda)$ with the properties described above.
\begin{itemize}
    \item[Case 1.] $\lambda$ is a successor ordinal, that is $\lambda=\alpha+1$ for some ordinal $\alpha$.

    Let $f_\lambda:[0,1]\rightarrow [0,1]$ be the continuous map given by
    \begin{equation*}
    f_\lambda(x)=\begin{cases}
        \frac{1}{2} f_\alpha (2x) & \quad\text{if }x\in [0,\frac{1}{2}]\\
        x+(x-\frac{1}{2})(x-1) & \quad\text{if }x\in (\frac{1}{2},1],
    \end{cases}
\end{equation*}
where $f_\alpha:[0,1]\to [0,1]$ is defined by the inductive hypothesis.
By properties (ii) and (iv) of the system $([0,1],f_\alpha)$, one obtains that the system $([0,\frac 1 2],f_\lambda |_{[0,\frac{1}{2}]})$ also verifies (ii) and (iv). 
Since $([0,1],f_\alpha)$ verifies (i), it follows that $0$ and $\frac 1 2$ are fixed points for the map $f_\lambda$.
Moreover, we have that $(\mathfrak{C}([0, \frac 1 2], f_\lambda |_{[0,\frac{1}{2}]}),\preceq)$ has the same order type of $(\mathfrak{C}([0,1],{f_\alpha}),\preceq)$ which is, by (iii), equal to $(\alpha+1,\in)$. 

On the other hand, if we consider the system $([\frac 1 2,1],f_\lambda |_{[\frac{1}{2},1]})$, we see that $\frac{1}{2}$ and $1$ are fixed points, $f_\lambda|_{(\frac{1}{2},1)}$ is strictly increasing and $f_\lambda|_{[\frac{1}{2},1]}(x)<x$ for every $x\in (\frac 1 2,1)$. 
This implies that 
$\mathfrak{C}_{f_\lambda |_{[\frac{1}{2},1]}}=\{ [\frac 1 2],[1] \}$ and, in particular, $[\frac{1}{2}]\prec [1]$ and $[1]\not\prec [\frac{1}{2}]$. 

Therefore, the system $([0,1],f_\lambda)$ is such that $\mathfrak{C}_{f_\lambda}=\mathfrak{C}_{f_\lambda |_{[0,\frac{1}{2}]}}\cup \{[1]\}$ and $[x]\prec [1]$ for every $[x]\in \mathfrak{C}_{f_\lambda |_{[0,\frac{1}{2}]}}$.
Therefore, the $\mathcal{C}$-poset $(\mathfrak{C}_{f_\lambda},\preceq)$ has order type $(\alpha+2,\in)$ and so the system verifies properties (i), (ii), (iii) and (iv). 

\item[Case 2.] $\lambda$ is a limit ordinal. 

By using the Cantor normal form, we can write $\lambda=\omega^{\alpha_1}+\ldots+\omega^{\alpha_k}$ where $k\in\mathbb{N}$ and $\alpha_1\ge \ldots\ge \alpha_k$ are countable ordinals. 
In particular, since $\lambda$ is a limit ordinal we have that $\alpha_k\ge1$, so we can always take $1\le \alpha<\lambda$ such that $\lambda=\alpha+\omega^{\alpha_k}$. 
Moreover, there exists a sequence of ordinals $\{\beta_j\}_{j\in\mathbb{N}}$ such that $\sup_{j\in\mathbb{N}}\beta_j=\omega^{\alpha_k}$.
Note that $\omega^{\alpha_k}$ is additively indecomposable, that is $\sum_{i=1}^j \beta_i<\omega^{\alpha_k}$ for every $j\in\mathbb{N}$. Then, we have:
\begin{equation}\label{prop_sup_ord}
  \sup_{j\in\mathbb{N}}\beta_j= \sup_{j\in\mathbb{N}}\sum_{i=1}^j \beta_i= \omega^{\alpha_k}.  
\end{equation}
Consider the strictly increasing sequence $\{a_n
\}_{n\in\mathbb{N}_0}$ in $[0,1]$ where $a_n=\frac{n}{n+1}$ and for every $n\in\mathbb{N}_0$ let us define a continuous function $f_{\lambda,n}:[a_n,a_{n+1}]\to [a_n,a_{n+1}]$ as follows.

Let $f_{\lambda,0}:[0,\frac{1}{2}]\to[0,\frac{1}{2}]$ be given by 
$$
f_{\lambda,0}(x)=\frac{1}{2}f_\alpha(2x),
$$
where $f_\alpha:[0,1]\to [0,1]$ is defined by the inductive hypothesis.
For every $n\in\mathbb{N}$, let us define the function $f_{\lambda,n}:[a_n,a_{n+1}]\to [a_{n},a_{n+1}]$ as 
\[
f_{\lambda,n}(x)=\frac{f_{\beta_n}((n+1)(n+2)(x-a_n))}{(n+1)(n+2)}+a_n,
\]
where $f_{\beta_n}:[0,1]\to [0,1]$ is defined by the inductive hypothesis for every $n\in \mathbb{N}$.

By inductive hypothesis, the systems $([0,1],f_\alpha)$ and $\{([0,1],f_{\beta_n})\}_{n\in\mathbb{N}}$ verify properties (i), (ii), (iii) and (iv).
Then, for every $n\in\mathbb{N}_0$ the points $a_n$ and $a_{n+1}$ are fixed points for the map $f_{\lambda,n}$ and the system $([a_n,a_{n+1}],f_{\lambda,n})$ verifies (ii) and (iv).

Let us now define the continuous function $f_\lambda:[0,1]\rightarrow[0,1]$ as follows:
\[
f_\lambda=\left(\bigcup_{n\geq 0}f_{\lambda,n}\right)\cup\{(1,1)\}.
\]
By construction, we have:
\begin{equation}\label{poset_transf}
  \mathfrak{C}([0,1],{f_\lambda})=\bigcup_{n\ge 0} \mathfrak{C}([a_n,a_{n+1}],f_{\lambda,n})\cup \{[1]\}.  
\end{equation}
Notice that the system $([0,1],f_\lambda)$ verifies properties (i) and (ii).

If we consider the system $([0,\frac 1 2],{f_{\lambda,0}})$, we see that the poset $(\mathfrak{C}_{f_{\lambda,0}},\preceq)$ has the same order type of $(\mathfrak{C}([0,1],{f_\alpha}),\preceq)$ which is, by (iii), equal to $(\alpha+1,\in)$.
Analogously, the poset $(\mathfrak{C}([a_n,a_{n+1}],{f_{\lambda,n}}),\preceq)$ has order type $(\beta_n+1,\in)$ for every $n\in\mathbb{N}$.  

Consider the system $([0,1],f_\lambda)$ and the poset $(\mathfrak{C}_{f_\lambda},\preceq)$ where $\mathfrak{C}_{f_\lambda}$ is described in \eqref{poset_transf}. 
Since for every $n\in\mathbb{N}_0$ the system $([a_n,a_{n+1}],f_{\lambda,n})$ verifies (ii) and (iv), if $[x]\in \mathfrak{C}_{f_{\lambda,n}}$ and $[y]\in \mathfrak{C}_{f_{\lambda,m}}$ with $0\le n<m$, then $[x]\preceq [y]$. 
In particular, we have that $[x]\prec [1]$ for all $[x]\in \bigcup_{n\ge 0} \mathfrak{C}_{f_{\lambda,n}}$.
Therefore, the system $([0,1],f_\lambda)$ verifies property (iv) and the poset $(\mathfrak{C}_{f_\lambda},\preceq)$ has order type 
$$
\left(\alpha+\sup_{j\in\mathbb{N}}\sum_{i=1}^j \beta_i+1,\in \right)=(\lambda+1,\in),
$$
where the last equality follows by \eqref{prop_sup_ord}.
\end{itemize}
\end{proof}
\noindent Notice that if the map $f$ of a dynamical system $(X,f)$ is a bijection, then the order type of the poset $(\mathfrak{C}_f,\preceq)$ is the order type of $(\mathfrak{C}_{f^{-1}},\preceq^{-1})$, where $\preceq^{-1}$ is the inverse relation of $\preceq$.
Moreover, $\mathfrak{C}_f = \mathfrak{C}_{f^{-1}}$ because for every $x, y \in X$ it holds that
\[
x\,\mathcal{C}\, y \text{ via $f$} \iff y\,\mathcal{C}\, x \text{ via $f^{-1}$}.
\]
\noindent Since the map $f_\lambda$ in Theorem~\ref{ordinal_order_type} is bijective, we immediately obtain the following corollary.

\begin{cor}\label{cor_on_bijective_function}
For every countable ordinal $ \lambda<\omega_1$, there is a continuous dynamical system $([0,1],g_\lambda)$ such that
the poset $(\mathfrak{C}_{g_\lambda},\preceq)$ has order type $(\lambda+1, \ni)$.

\end{cor}
\begin{proof} Since $f^{-1}_\lambda$ is continuous, considering $g_\lambda=f^{-1}_\lambda$ the statement of the corollary follows.
\end{proof}

\subsection{Non-compact spaces}
The compactness assumption in Theorem~\ref{main} plays a central role in constraining the structure of the ${\mathcal C}$-poset. It is interesting to understand whether similar obstructions hold when the phase space is no longer compact. This section is motivated by the desire to test the robustness of the obstruction under weaker topological assumptions.

\noindent Indeed, under certain conditions, the existence of a complete Lyapunov function (as guaranteed by Theorem~\ref{conley_th} for compact metric spaces) can be extended to more general non-compact settings (see \cite{hu91, hu92, hu98}).

\noindent In what follows, we focus on the class of locally compact and second countable metric spaces (see Def.~\ref{def_noncompact}). In this setting, we recover the impossibility of realizing a countable dense linear order under mild additional assumptions. This result is derived from a theorem by M. Hurley in \cite{hu91}.

\begin{thm}\label{th_cc_loccomp} Let $(X,d)$ be a locally compact, second countable metric space and let $f:X\to X$ be a continuous map. Let us suppose that at least one of the following two conditions hold:
\begin{enumerate}
\item[1.] $CR_f$ is strongly invariant, that is $f(CR_f)=CR_f$,
\item[2.] for each natural number $k\geq 1$,
\[
CR_f=\bigcup_{n\geq k}CR_{f^n},
\]
where $f^n=f\circ f\circ\ldots\circ f$ ($n$ times).
\end{enumerate}
Then the poset $(\mathfrak{C}_f,\preceq)$ cannot be linearly and densely ordered.
\end{thm}
\noindent The proof is analogous to the proof of Theorem \ref{main} because each of conditions 1 and 2 implies the existence of a complete Lyapunov function for $f$ that verifies property (iv$'$) (cf. \cite[Theorem~4.0, p.~720, and Theorem~8.2, p.~727]{hu91}).

\noindent We note that conditions 1 and 2 of Theorem~\ref{th_cc_loccomp} are implied by the conditions mentioned in (a)-(c) as described below. For a full discussion of this, see \cite[Theorem 4.0]{hu91}.
Let $f$ be a continuous function on a metric space $X$, then:
\begin{itemize}
\item[(a)] if $f$ is a homeomorphism, then $f$ is a proper map, which in turn implies condition 1;
\item[(b)] if $f$ is uniformly continuous, then
\[
CR_f=CR_{f^n} \text{ for all $n\geq 1$},
\]
from which condition 2 follows;
\item[(c)] if $X$ is locally compact and $CR_f$ is compact, then each chain component of $f$ is compact, which in turn implies condition 1.
\end{itemize}

\section{$\mathcal{C}$-posets in discontinuous systems}

In recent years increasing attention has been devoted to dynamical properties of highly discontinuous maps. The motivation for this research line has been mainly theoretical rather than applicative, and an important goal has been that of clarifying which specific properties (often weaker than full continuity) are used in proving classical results in topological dynamics. Generally, the focus is on some specific weakening of continuity rather than on the general case where no regularity assumption is made at all.
For instance, in the literature it has been considered the case of systems whose time evolution is only piecewise continuous (\cite{llibre2015birth,llibre2015maximum}), belongs to some low-order Baire class (\cite{steele17, steele18, steele19, steelebaire2}), or is quasicontinuous (\cite{crannellmartelli,crannel05}), Darboux (\cite{pawlakdarboux,pawentropy}), almost continuous (\cite{pawlak19,szuca10}). 
Among the works dealing with the general case, we mention \cite{ciklova}. 
For what concern specifically chain recurrence, it has been investigated for Baire class~1 and, more generally, Baire class~$\alpha$ maps (\cite{alikhani2019chain,alikhani2022iterates}).

\noindent In the following, we establish some results concerning discontinuous dynamical systems. First of all, we construct interval discontinuous systems whose $\mathcal{C}$-poset is countable and densely ordered, either with or without a maximum element (Theorems \ref{order_type_[0,1]} and \ref{order_nomin}). Then we prove that, without any regularity assumption on the map, any countable linearly ordered subset of the $\mathcal{C}$-poset must have a lower bound (Lemma \ref{lem_max} and Theorem \ref{thm_max}) and that the $\mathcal{C}$-poset has a minimal element (Remark~\ref{rem_minimal}).

\begin{defi}\label{def_interm}
Let $A,B\subseteq\mathbb{R}$ be two nonempty disjoint connected sets. 
We say $A<B$, if $x<y$ for all $x\in A$ and $y\in B$. 
We say that $C\subseteq \mathbb{R}$ nonempty connected set is \emph{intermediate} between $A$ and $B$, if $A<C<B$.
Notice that the sets $A$, $B$ and $C$ can also be a singleton.
\end{defi}

\noindent Let us define the following partial order relation on the power set of the unit interval $[0,1]$. 

\begin{defi}\label{def_order_power}
    Let $\mathcal{P}([0,1])$ be the power set of the unit interval. 
    Let us define a partial order relation on $\mathcal{P}([0,1])$, denoted by $\le_{\mathcal{P}}$, as follows.
    For $I,J\in \mathcal{P}([0,1])$, we say that $I\leq_{\mathcal{P}} J$ if and only if either $I=J$ or $x<y$ for every $x\in I$ and $y\in J$. We say that $I <_\mathcal{P} J$ if $I\le_\mathcal{P} J$ and $I\neq J$.
\end{defi}

\begin{thm}\label{order_type_[0,1]}
There is a discontinuous dynamical system $([0,1],f)$ such that the poset $(\mathfrak{C}_f,\preceq)$ has order type $([0,1]\cap\mathbb{Q},\le)$.
\end{thm}
\begin{proof}
We want to define a dynamical system $([0,1],f)$ such that its chain components poset is order isomorphic to $(L,\leq_{\mathcal{P}})$, which is in turn order isomorphic to $([0,1]\cap \mathbb{Q}, \leq)$, where $L=\{[a_n,b_n]\}_{n\geq 0}$ is a countable set of pairwise disjoint intervals in $[0,1]$ ordered by the order relation $\leq_{\mathcal{P}}$ given by Def.~\ref{def_order_power}.

\noindent Set $[a_0,b_0]:=\left[0,\frac{1}{4}\right]$ and $[a_1,b_1]:=\left[\frac{3}{4},1\right]$. 
Let us define
\begin{align*}
L_0&:=\{[a_0,b_0],[a_1,b_1]\}=\left \{\left[0,\frac{1}{4}\right], \left[\frac{3}{4},1\right]\right \}\\
\\
L_1&:=\left\{[a_0,b_0], \left[a_2, b_2\right], [a_1,b_1]\right\}=L_0\cup\{[a_2,b_2]\},\\
\end{align*}
where $[a_2,b_2]$ is intermediate between $[a_0,b_0]$ and $[a_1,b_1]$. Notice that $[a_0,b_0] <_\mathcal{P} [a_2,b_2]<_\mathcal{P} [a_1,b_1]$.
Let us proceed by setting
\[
L_2:=\left\{[a_0,b_0],[a_3, b_3], [a_2,b_2],[a_4,b_4],[a_1,b_1]\right\}=L_1\cup\{[a_3,b_3],[a_4,b_4]\}
\]
where 
\begin{itemize}
\item $[a_3,b_3]$ is intermediate between $[a_0,b_0]$ and $[a_2,b_2]$, 
\item $[a_4,b_4]$ is intermediate between $[a_2,b_2]$ and $[a_1,b_1]$.
\end{itemize}
So, we see that $[a_0,b_0]<_\mathcal{P}[a_3,b_3]  <_\mathcal{P} [a_2,b_2]<_\mathcal{P}[a_4,b_4]<_\mathcal{P} [a_1,b_1]$.

\noindent Proceeding inductively, for $n\ge 2$, given $L_n=\{I_0, I_1,\cdots , I_{2^n}\}$ with $I_0<_\mathcal{P} I_1<_\mathcal{P}\cdots <_\mathcal{P} I_{2^n}$, we set
\[
L_{n+1}:=L_n\cup\{I_{2^n+1}, I_{2^n+2},\cdots , I_{2^{n+1}}\}
\]
where $I_{2^n+k}=[a_{2^n+k}, b_{2^n+k}]$ is intermediate between $I_{k-1}$ and $I_{k}$ for every $k\in\{1,\ldots,2^n\}$.
So, it holds that $I_{k-1} <_\mathcal{P} I_{2^n+k} <_\mathcal{P} I_k$ for every $k\in\{1,\ldots,2^n\}$.
Notice that $\left|L_n\right|=2^n+1$ for every $n\geq 0$, and that $L_n\subsetneq L_m$ whenever $n<m$.
Set 
$$
A_n:=[0,1]\setminus \cup L_n.
$$ 
Observe that $A_n$ is an open set and $A_{n+1}\subseteq A_{n}$ for all $n\in\mathbb{N}_0$. Moreover, we can choose the intervals $\{[a_n,b_n]\}_{n>1}$ so that 
\begin{equation}\label{eq_mis_compl}
    \lim_{n\to\infty} m(A_n)=0,
\end{equation}
where $m$ indicates the Lebesgue measure on $[0,1]$.

\noindent Finally, we set 
\[
L:=\bigcup_{n\geq 0}L_n=\{[a_n,b_n]\}_{n\in\mathbb{N}_0}.
\]
Let us now define $f:[0,1]\to [0,1]$ as:
\begin{equation}\label{dis_func}
    f(x)=
    \begin{cases}
        a_n &\quad\text{if } x\in[a_n,b_n] \text{ for some }n\in\mathbb{N}_0\\
        0 &\quad\text{if } x\in[0,1]\setminus \cup L
    \end{cases}
\end{equation}
From the construction of $L$ it is clear that $(L,\le_\mathcal{P})$ is linearly ordered and for every $I, J\in L$ such that $I<_\mathcal{P} J$, there is $K\in L$ such that
\[
I<_\mathcal{P} K<_\mathcal{P} J,
\]
which means that $(L,<_\mathcal{P})$ is also densely ordered.
Since $[a_0,b_0]<_\mathcal{P} [a_n,b_n]<_\mathcal{P} [a_1,b_1]$ for every $n\ge2$, we have that $(L,\le_\mathcal{P})$ is order isomorphic to $([0,1]\cap \mathbb{Q}, \leq)$.

\noindent It remains to show that $(\mathfrak{C}_f,\preceq)$ is order isomorphic to $(L,\le_\mathcal{P})$.
Notice that the points $a_n$ for every $n\in\mathbb{N}_0$ are the only fixed points for the map $f$.
We want to prove that
$$
\mathfrak{C}_f=
\{\{a_n\}: n\in\mathbb{N}_0\}.
$$
\noindent Let us show that for every $[a_p,b_p],[a_q,b_q]\in L$ with $[a_p,b_p]<_\mathcal{P} [a_q,b_q]$ it holds that $[a_p]\prec [a_q]$.
\begin{enumerate}
\item Let us first prove that $[a_p]\preceq [a_q]$, that is, $a_q\,\mathcal{C}\, a_p$.   

Fix $\varepsilon>0$.
If $\varepsilon> a_q-b_p$ then the three points $a_q, b_p, a_p$ form an $\varepsilon$-chain from $a_q$ to $a_p$.
Indeed, $d(f(a_q),b_p)=d(a_q,b_p)<\varepsilon$ and $f(b_p)=a_p$.
So suppose that $\varepsilon<a_q-b_p$. 
By \eqref{eq_mis_compl} and since $(L,<_\mathcal{P})$ is densely ordered, there is $n$ sufficiently large that we can choose $n-2$ intervals 
\[
J_{n-2}<_\mathcal{P} J_{n-1}<_\mathcal{P}\cdots <_\mathcal{P} J_1
\]
in $L$ with $J_1=[x_1,y_1]$, $J_2=[x_2,y_2]$, $\ldots$, $J_{n-2}=[x_{n-2},y_{n-2}]$, such that
\begin{align*}
    &d(a_q,y_1)<\varepsilon\\
    &d(x_1,y_2)<\varepsilon\\
    &\quad\quad\vdots\\
    &d(x_{n-3},y_{n-2})<\varepsilon\\
    &d(x_{n-2},b_p)<\varepsilon
\end{align*}
From the definition of $f$ given by \eqref{dis_func}, we have that $f(a_q)=a_q$ and $f(y_i)=x_i$ for every $i=1,\ldots,n-2$. Therefore, we can rewrite the previous inequalities as
\begin{align*}
    &d(f(a_q),y_1)<\varepsilon\\
    &d(f(y_1),y_2)<\varepsilon\\
    &\quad\qquad\vdots\\
    &d(f(y_{n-3}),y_{n-2})<\varepsilon\\
    &d(f(y_{n-2}),b_p)<\varepsilon
\end{align*}
Since $f(b_p)=a_p$, it follows that the points
$$
a_q,y_1,y_2,\ldots,y_{n-2},b_p, a_p
$$
form an $\varepsilon$-chain from $a_q$ to $a_p$. 
We conclude that $[a_p]\preceq [a_q]$.

\item It remains to show that $[a_q]\not\preceq [a_p]$.
Notice that, by \eqref{dis_func} we have $f(x)\le x$ for every $x\in [0,1]$.
Fix $\varepsilon>0$ such that $\varepsilon< b_p-a_p\leq\frac{1}{4}$.
Then, there can be no $\varepsilon$-chain 
\[
a_p=x_0,x_1,x_2,\ldots,x_n=a_q
\]
from $a_p$ to $a_q$.
In fact, suppose by contradiction that there is an $\varepsilon$-chain  $x_0,\ldots,x_n$ from $a_p$ to $a_q$. If $x_1\ge x_0=a_p$, since 
\begin{equation}\label{d(f(x_0),x_1)<epsilon}
d(f(x_0),x_1)<\varepsilon,
\end{equation}
then $x_1\in [a_p,b_p]$. Indeed, since $f(x_0)=f(a_p)=a_p$, inequality (\ref{d(f(x_0),x_1)<epsilon}) becomes $d(a_p,x_1)<\varepsilon$.
And again, if $x_2\ge x_0$, the condition 
$$
d(f(x_1),x_2)=d(a_p,x_2)<\varepsilon
$$
implies that $x_2\in [a_p,b_p]$.
So, let $1\le j\le n-1$ be the first index for which 
$$
x_j< f(x_{j-1})=a_p=x_0.
$$
If $x_j\in [0,1]\setminus \cup L$, then $f(x_j)=0$ and since $\varepsilon<\frac 1 4$, it follows that 
$$
f(x_{k})=0 \quad \forall\, k=j,\ldots,n-1.
$$ 
Hence $a_q\in[0,\frac 1 4]$, which is impossible.
If $x_j\in L$, since $f(x_j)\le  x_j< a_p$, then $x_k < b_p$ for every $j<k\le n$. Therefore, $d(f(x_{n-1}),a_q)\ge d(a_p,a_q)>\varepsilon$ which is a contradiction.
\end{enumerate}

\noindent In particular, from the definition of the map $f$ given by \eqref{dis_func} and by points 1. and 2. we have that
$
\mathfrak{C}_f=
\{\{a_n\}:n\in\mathbb{N}_0\}
$.

\noindent Finally, by points 1. and 2. it follows that the application
$$
[a_n,b_n]\mapsto \{a_n\} \qquad (n\in\mathbb{N}_0)
$$
is an order isomorphism between $(L,\le_\mathcal{P})$ and $(\mathfrak{C}_f,\preceq)$.
\end{proof}

\noindent The theorem proved above shows that it is possible to realize a countable linearly ordered and densely ordered set with endpoints as the chain components poset of a discontinuous dynamical system $(X,f)$ with $X$ compact metric space.
If we modify the construction at step zero, that is, when we define the first set in $L$, then it is possible to realize a countable linearly ordered and densely ordered set whose order type is $([0,1)\cap\mathbb{Q}, \leq_{\mathbb{Q}})$.

\begin{thm}\label{order_nomin}
There is a discontinuous dynamical system $([0,1],f)$ such that the poset $(\mathfrak{C}_f,\preceq)$ has order type $([0,1)\cap\mathbb{Q},\leq)$.
\end{thm}
\begin{proof} 
Similarly to what was done in the proof of Theorem \ref{order_type_[0,1]}, we want to define a dynamical system $([0,1],f)$ such that its chain components poset is order isomorphic to $(L,\leq_\mathcal{P})$, which is in turn order isomorphic to $([0,1)\cap \mathbb{Q}, \leq)$, where $L=\{[a_n,b_n]\}_{n\in\mathbb{N}}$ is a suitable countably infinite set of pairwise disjoint intervals in $[0,1]$.
We start by setting $a_1:=0$, and
\[
L_1:=\{[a_1,b_1], [a_2,b_2]\}=\{[0,b_1], [a_2,b_2]\}
\]
where the interval $[a_2,b_2]$ is intermediate between $[a_1,b_1]$ and $1$.
Let us proceed by setting
\[
L_2:=\{[a_1,b_1], [a_3,b_3],[a_2,b_2], [a_4,b_4]\}=L_1\cup\{[a_3,b_3],[a_4,b_4]\},
\]
where 
\begin{itemize}
\item $[a_3,b_3]$ is intermediate between $[a_1,b_1]$ and $[a_2,b_2]$, and
\item $[a_4,b_4]$ is intermediate between $[a_2,b_2]$ and 1.
\end{itemize}
Notice that $[a_1,b_1]<_\mathcal{P} [a_3,b_3] <_\mathcal{P} [a_2,b_2] <_\mathcal{P} [a_4,b_4] <_\mathcal{P} 1$.

\noindent Proceeding inductively, for $n\ge2$, given $L_n=\{I_1, I_2,\cdots , I_{2^n}\}$ with $I_1<_\mathcal{P} I_2<_\mathcal{P}\cdots <_\mathcal{P} I_{2^n}$, we set
\[
L_{n+1}:=L_n\cup\{I_{2^n+1}, I_{2^n+2},\cdots , I_{2^{n+1}}\}
\]
where
\begin{itemize}
    \item $I_{2^n+k}=[a_{2^n+k}, b_{2^n+k}]$ is intermediate between $I_k$ and $I_{k+1}$, so that $I_k<_\mathcal{P} I_{2^n+k}<_\mathcal{P} I_{k+1}$ for every $k\in\{1,\ldots,2^n-1\}$,
    \item $I_{2^{n+1}}=[a_{2^{n+1}}, b_{2^{n+1}}]$ is intermediate between $I_{2^n}$ and $1$, so that $I_{2^n}<_\mathcal{P} I_{2^{n+1}}<_\mathcal{P} 1$.
\end{itemize}
For every $n\in\mathbb{N}$ we set 
$$
A_n:=[0,1]\setminus \cup L_n.
$$ 
Notice that $A_n$ is an open set and $A_{n+1}\subseteq A_{n}$ for all $n\in\mathbb{N}$. Moreover, we can choose the intervals $\{[a_n,b_n]\}_{n>1}$ so that 
\begin{equation*}
    \lim_{n\to\infty} m(A_n)=0,
\end{equation*}
where $m$ indicates the Lebesgue measure on $[0,1]$.
Finally, we set 
$$
L:=\bigcup_{n\ge 1}L_n=\{[a_n,b_n]\}_{n\in\mathbb{N}}.
$$
Let us now define the function $f:[0,1]\to [0,1]$ as:
\begin{equation*}\label{disc_nomin_f}
    f(x)=
    \begin{cases}
        a_n &\quad \text{if $x\in [a_n,b_n]$ for some $n\in\mathbb{N}$}\\
        0 &\quad   \text{if $x\in [0,1]\setminus \cup L$}
    \end{cases}
\end{equation*}
\noindent From the construction of $L$ it is clear that $(L,\le_\mathcal{P})$ is linearly ordered and for every $I, J\in L$ such that $I<_\mathcal{P} J$, there is $K\in L$ such that
$ I<_\mathcal{P} K<_\mathcal{P} J$, which means that $(L,<_\mathcal{P})$ is also densely ordered.
Since $[a_1,b_1]<_\mathcal{P}[a_n,b_n]$ for all $n\ge2$ and $L$ does not have a maximum element with respect to $<_\mathcal{P}$, it follows that $(L,\le_\mathcal{P})$ is order isomorphic to $([0,1)\cap \mathbb{Q}, \leq)$.
The proof that $\mathfrak{C}_f=\{\{a_n\}:n\in\mathbb{N}\}$ and that the chain components poset $(\mathfrak{C}_f,\preceq)$ is order isomorphic to $(L,\le_\mathcal{P})$ is analogous to that of Theorem \ref{order_type_[0,1]}.
\end{proof}

\noindent Note that in both the constructions of the two previous theorems, the map is not bijective. 
In particular, in Theorem~\ref{order_nomin}, this non-bijectivity is necessary: as stated above immediately before Corollary~\ref{cor_on_bijective_function}, a bijective map would allow us to consider its inverse, yielding a realization of a $\mathcal{C}$-poset that has order type $((0,1] \cap \mathbb{Q},\le )$, which is impossible in a compact setting due to the lack of minimal elements: as shown in Remark~\ref{rem_minimal}, any $\mathcal{C}$-poset arising from a compact metric space must contain at least one minimal element. 

\noindent The two order types have been realized with compact metric spaces and discontinuous functions.
It is straightforward to realize order types $[(0,1]\cap\mathbb{Q},\leq)$ and $((0,1)\cap\mathbb{Q},\leq)$ by discontinuous dynamical systems $(X,f)$ with $X$ a non-compact metric space as shown in the following example.

\begin{exa}
We want to construct two discontinuous dynamical systems $((0,1),f)$ and $((0,1],g)$ such that their chain components posets have, respectively, the same order types of $((0,1)\cap\mathbb{Q},\leq)$ and $((0,1]\cap\mathbb{Q},\leq)$.

\noindent Let us describe the construction for the order type $((0,1)\cap\mathbb{Q},\leq)$, an analogous argument can be used for the order type $((0,1]\cap\mathbb{Q},\leq)$. 

\noindent We want to define a set of pairwise disjoint intervals $L=\{[a_n,b_n]\}_{n\in\mathbb{N}_0}$ in $[0,1]$ such that $(L,\le_\mathcal{P})$ is order isomorphic to $((0,1)\cap\mathbb{Q},\leq)$. 
Set
$$
L_0:=\{[a_0,b_0]\},
$$
with $0<a_0<b_0<1$. Then, we define $L_1:=L_0\cup \{[a_1,b_1],[a_2,b_2]\}$ where
\begin{itemize}
    \item the interval $[a_1,b_1]$ is intermediate between 0 and $[a_0,b_0]$,
    \item the interval $[a_2,b_2]$ is intermediate between $[a_0,b_0]$ and 1. 
\end{itemize}
So we see that $0<_\mathcal{P} [a_1,b_1] <_\mathcal{P} [a_0,b_0] <_\mathcal{P} [a_2,b_2] <_\mathcal{P} 1$.

\noindent Proceeding inductively, for $n\ge1$, given 
$$
L_n=\{I_0,I_1,\ldots,I_{2^{n+1}-2}\}
$$ 
with $I_0<_\mathcal{P} I_1<_\mathcal{P} \ldots <_\mathcal{P} I_{2^{n+1}-2}$, we set 
\[
L_{n+1}:=L_n\cup\{I_{2^{n+1}-1}, I_{2^{n+1}},\cdots , I_{2^{n+2}-2}\}
\]
where
\begin{itemize}
\item $I_{2^{n+1}-1}=[a_{2^{n+1}-1}, b_{2^{n+1}-1}]$ is intermediate between $0$ and $I_0$,
\item $I_{2^{n+1}+k}=[a_{2^{n+1}+k}, b_{2^{n+1}+k}]$ is intermediate between $I_k$ and $I_{k+1}$ for every $k\in\{0,\ldots,2^{n+1}-3\}$,
\item $I_{2^{n+2}-2}=[a_{2^{n+2}-2}, b_{2^{n+2}-2}]$ is intermediate between $I_{2^{n+1}-2}$ and $1$.
\end{itemize}
So we have $0<_\mathcal{P} I_{2^{n+1}-1} <_\mathcal{P} I_0$,
$I_k<_\mathcal{P} I_{2^{n+1}+k} <_\mathcal{P} I_{k+1}$ for every $k\in\{0,\ldots,2^{n+1}-3\}$, and $I_{2^{n+1}-2}<_\mathcal{P} I_{2^{n+2}-2} <_\mathcal{P} 1$.
For every $n\in\mathbb{N}_0$ we set 
$$
A_n:=[0,1]\setminus \cup L_n.
$$ 
Notice that $A_n$ is an open set and $A_{n+1}\subseteq A_{n}$ for all $n\in\mathbb{N}_0$. Moreover, we can choose the intervals $\{[a_n,b_n]\}_{n>0}$ so that 
\begin{equation*}
    \lim_{n\to\infty} m(A_n)=0,
\end{equation*}
where $m$ indicates the Lebesgue measure on $[0,1]$.
Finally, we set
\[
L:=\bigcup_{n\geq 0}L_n.
\]
Let us now define the function $f:(0,1)\to (0,1)$ as follows:
\begin{equation*}
    f(x)=
    \begin{cases}
        a_n \quad & \text{if $x\in [a_n,b_n]$ for some $n\in\mathbb{N}_0$}\\
        \frac{1}{n+2} \quad & \text{if $x\in [\frac{1}{n+1},\frac{1}{n})\setminus \cup L$ for some $n\in\mathbb{N}$}
    \end{cases}
\end{equation*}

\noindent The construction of $L$ ensures that there are infinitely many indices $n_1,n_2,\ldots$ such that
\[
[a_{n_1},b_{n_1}]<_\mathcal{P} [a_{n_2},b_{n_2}]<_\mathcal{P} \cdots
\]
with $\lim_{i\to\infty}a_{n_i}=1$, and there are infinitely many indices $m_i$ such that
\[
[a_{m_1},b_{m_1}]>_\mathcal{P} [a_{m_2},b_{m_2}]>_\mathcal{P}\cdots
\]
with $\lim_{i\to\infty}a_{m_i}=0$.
Moreover, it is clear that for all $I,J\in L$ such that $I<_\mathcal{P} J$ there is an interval $K\in L$ such that $I<_\mathcal{P} K<_\mathcal{P} J$. Then $(L,\le_\mathcal{P})$ is linearly and densely ordered and in particular, since $L$ does not have a minimum and a maximum element, it is order isomorphic to $((0,1)\cap\mathbb{Q},\leq)$.
The proof that $\mathfrak{C}_f=\{\{a_n\}: n\in\mathbb{N}_0\}$ and that the chain components poset $(\mathfrak{C}_f,\preceq)$ is order isomorphic to $(L,\le_\mathcal{P})$ is analogous to that of Theorem \ref{order_type_[0,1]}.
\end{exa}

\noindent We have seen that in non-compact metric spaces it is possible to construct dynamical systems whose chain components poset has no minimal elements.
We now show that this is no longer the case when the space is compact: the existence of at least one minimal element becomes inevitable.

\noindent Let us state the following fact (for a proof, see \cite[Theorem 3.1]{chainsregularity}):
\begin{thm}
\label{thm1}
Let $X$ be a metrizable space and $f:X\to X$ a map. If $X$ is compact, then $CR_f\ne\emptyset$.
\end{thm}

\noindent The following result is an easy consequence of the previous theorem (for a proof, see \cite[Theorem 3.4]{chainsregularity}).
\begin{thm}
\label{cor1}
Let $X$ be a metrizable space and $f:X\to X$ a map. If $X$ is compact, then for every $x\in X$ there is $y\in X$ such that $y$ is chain recurrent and $x\,\mathcal{C}\, y$.
\end{thm}
\noindent We can now prove our final result, for which most of the work is done by the following lemma.
\begin{lem}\label{lem_max}
Let $X$ be a compact metric space and $f:X\rightarrow X$ a map. If $\{x_k\}_{k\in\mathbb{N}}\subseteq CR_f$ is a sequence of chain-recurrent points such that, for every $k$, we have $x_k\,\mathcal{C}\, x_{k+1}$, then there is a chain-recurrent point $y$ such that $x_k\,\mathcal{C}\,y$ for every $k$.
\end{lem}
\begin{proof}
Since $X$ is compact, we have that the sequence $\{x_k\}_{k\in\mathbb{N}}$ has at least one limit point $z$. 
Fix $h \in\mathbb{N}$ and take $\varepsilon>0$. 
Then, there is a positive integer $\overline{k}>h$ such that $d(x_{\overline{k}},z)<\frac{\varepsilon}{2}$. 
By hypothesis and by transitivity of $\mathcal{C}$, there exists an $\frac{\varepsilon}{2}$-chain $x_h=z_0, z_1,\ldots, z_n=x_{\overline{k}}$ from $x_h$ to $x_{\overline{k}}$. 
Since
$$
d(f(z_{n-1}),z)\le d(f(z_{n-1}),x_{\overline{k}})+d(x_{\overline{k}},z)<\frac \epsilon 2 +\frac \epsilon 2=\varepsilon,
$$
we have that the points $z_0,z_1,\ldots,z_{n-1},z$ form an $\varepsilon$-chain from $x_h$ to $z$. 
Therefore, $x_h \,\mathcal{C}\, z$, and by the arbitrariness of $h$, it follows that $x_k\,\mathcal{C}\, z$ for every $k\in\mathbb{N}$.
By Theorem \ref{cor1}, there is $y\in CR_f$ such that $z\,\mathcal{C}\,y$. Therefore, by the transitivity of $\mathcal{C}$, it follows that $x_k\,\mathcal{C}\, y$ for every $k\in\mathbb{N}$.
\end{proof}

\begin{rem}\label{rem_net}
The argument of Lemma \ref{lem_max} extends from sequences to nets (see Def.~\ref{def_netcluster}). 
Indeed, let $\Lambda$ be a directed set (see Def.~\ref{def_directedset}) and
$\{x_\lambda\}_{\lambda\in\Lambda}\subseteq CR_f$ a net in $X$ such that $x_\lambda\,\mathcal{C}\,x_\mu$ whenever $\lambda\leq\mu$. We show that there exists $y\in CR_f$ such that $x_\lambda\,\mathcal{C}\, y$ for every $\lambda\in\Lambda$.
Since $X$ is compact, the net has a cluster point $z\in X$ (see Def.~\ref{def_netcluster}). 
Fix $\lambda_0\in\Lambda$ and $\varepsilon>0$.
Choosing $\mu\geq\lambda_0$ with $d(x_\mu,z)<\varepsilon/2$, exactly as in the
proof of Lemma \ref{lem_max} one obtains $x_{\lambda_0}\,\mathcal{C}\,z$.
Hence $x_\lambda\,\mathcal{C}\,z$ for every $\lambda\in\Lambda$. 
By
Theorem \ref{cor1}, there exists $y\in CR_f$ such that $z\,\mathcal{C}\,y$,
and therefore $x_\lambda\,\mathcal{C}\,y$ for every $\lambda\in\Lambda$.
\end{rem}

\noindent The following theorem follows straightforwardly from the previous lemma.

\begin{thm}\label{thm_max}
Let $X$ be a compact metric space and $f:X\rightarrow X$ a map.
If $\{C_k\}_{k\in\mathbb{N}}\subseteq \mathfrak{C}_f$ is a collection of distinct chain components such that, for every $k$, we have $C_{k+1} \prec C_k$, then there is a chain component $C$ such that $C\prec C_k$ for every $k$.
\end{thm}
\begin{proof}
Take a sequence of points $\{x_k\}_{k\in\mathbb{N}}$ such that $x_k\in C_k$, let $y$ be one of its sublimits and apply Lemma \ref{lem_max} to this sequence, so as to obtain that
\begin{equation}\label{foreveryk}
    \forall k\in\mathbb{N},\  x_k\,\mathcal{C}\,y. 
\end{equation}
Observe now that $y$ cannot belong to any of the components $C_k$, as otherwise, for $h>k$, we would have $x_{h}\,\mathcal{C}\, x_k$, so that $C_k=C_h$ against the assumptions. Therefore, $y$ belongs to a chain component $C$ such that, for every $k\in\mathbb{N}$, $C\neq C_k$. It follows then from Eq.\eqref{foreveryk} that $C\prec C_k$ for every $k$.
\end{proof}

\begin{rem}\label{rem_minimal}
The argument of Theorem~\ref{thm_max} extends to any linearly ordered subset of $\mathfrak{C}_f$. 
Indeed, let $\Lambda\subseteq \mathfrak{C}_f$ be a nonempty linearly ordered subset (in the inherited order), and regard $\Lambda$ as a directed set $(\Lambda,\le)$ where $\le$ is the inverse order $\preceq^{-1}$. 
For every chain component $C\in\Lambda$, choose $x_C\in C$. 
If $C\leq D$ in the directed order, then
$D\preceq C$, hence $x_C\,\mathcal{C}\,x_D$. By Remark~\ref{rem_net}, there exists
$y\in CR_f$ such that $x_C\,\mathcal{C}\,y$ for every $C\in\Lambda$.
Therefore, if $C_*=[y]$, then $C_*\preceq C$ for every $C\in\Lambda$.
Thus every nonempty linearly ordered subset in $(\mathfrak{C}_f,\preceq)$ has a lower bound.
Since $\mathfrak{C}_f\neq\emptyset$ by Theorem \ref{thm1}, Zorn's lemma
applied to the inverse order on $\mathfrak{C}_f$ shows that
$(\mathfrak{C}_f,\preceq)$ has a minimal element.
\end{rem}

\section{$\mathcal{C}^+$-posets under topological conjugacy}

In the classical setting, the chain relation ${\cal C}$ is known to be a topological invariant: if two systems are topologically conjugate, then their associated posets of chain components are order-isomorphic. In this final section, we explore whether a similar invariance holds for the refined relation ${\cal C}^+$ defined using continuous pointwise precision functions \(\varepsilon : X \to (0, +\infty)\).

\noindent We show that the poset structure $(\mathfrak{C}_f^+, \preceq^+)$, derived from the ${\cal C}^+$-chain components of a system $(X,f)$, is likewise preserved under topological conjugacy. This result highlights the robustness of the dynamical order structure even in the more flexible framework introduced by ${\cal C}^+$, and confirms that the refinement maintains a fundamental compatibility with the topological nature of the dynamics.

\noindent Recall that two dynamical systems $(X,f)$ and $(Y,g)$ are \emph{topologically conjugate} if there is a homeomorphism $h:X\rightarrow Y$ such that $h\circ f=g\circ h$.

\noindent Let us give the following stronger definition of the chain relation where we replace the fixed $\varepsilon$'s of Section \ref{sec:1} by arbitrary positive continuous functions. 
\begin{defi} 
Let $(X,f)$ be a dynamical system. We indicate by 
\[
C(X,\mathbb{R}^+)=\{\varepsilon:X\to(0,\infty):\text{ $\varepsilon$ is continuous}\}.
\]
We say that two points $x$ and $y$ of $X$ are ${C(X,\mathbb{R}^+)}$-{\em chain related}, in short $x\,\mathcal{C}^+\,y$, if for every $\varepsilon\in C(X,\mathbb{R}^+)$ there is a finite set of points $x_0,x_1,\ldots,x_n$ in $X$, with $n\in\mathbb{N}$, such that:
\begin{itemize}
\item $x_0=x$ and $x_n=y$,
\item $d(f(x_i),x_{i+1})<\varepsilon(f(x_i))$.
\end{itemize}
\end{defi}
\noindent The notion $\mathcal{C}^+$ is stronger than $\mathcal{C}$ in the sense that, for every $x,y\in X$, it holds that 
\begin{equation}\label{weaker}
x\,\mathcal{C}^+\,y \implies x\,\mathcal{C}\,y.
\end{equation}
The converse of \eqref{weaker} is true for compact metric spaces, while it is not necessarily true for non-compact, locally compact metric spaces (see Example 1 in \cite[p.~1141]{hu92}).

\noindent Set
$$
CR^+_f:=\{x\in X:x\,\mathcal{C}^+\,x\},
$$
and let $\mathfrak{C}^+_f:=CR^+_f/E^+$ be the set of equivalence classes on $CR^+_f$ with respect to the relation $E^+$ defined as
$$
x\, E^+\, y \iff x\,\mathcal{C}^+\, y \text{ and } y\,\mathcal{C}^+\, x.
$$
Finally, let $\preceq^+$ be the partial order relation on $\mathfrak{C}^+_f$ defined as
$$
[x]\preceq^+ [y] \iff y\,\mathcal{C}^+\, x.
$$
\\
Recalling the Def.~3.\textit{bis}., the chain relation $\mathcal{C}$ is a topological concept, in the sense that $x\,\mathcal{C}\,y$ depends only on the topology on $X$ and not on the metric, if of course the metric is compatible with the topology. 
Therefore, it is clear that if we consider two topologically conjugate dynamical systems $(X,f)$ and $(Y,g)$ with $X$ and $Y$ topological spaces, then the chain components posets $(\mathfrak{C}_f,\preceq)$ and $(\mathfrak{C}_g,\preceq)$ are order isomorphic.
Notice that the converse is not generally true, as shown elementarily by the example $X=Y=[0,1]$ with $f(x)=x$ and $g(x)=1-x$. 
In fact, both posets $(\mathfrak{C}_f,\preceq)$ and $(\mathfrak{C}_g,\preceq)$ have only one element, namely, $\mathfrak{C}_f=\mathfrak{C}_g=\{[0,1]\}$, so they are order isomorphic.
However, any homeomorphism $h$ between $(X,f)$ and $(Y,g)$ must satisfy $h(x)=1-h(x)$, which implies that $h(x)=\frac 1 2 $ for every $x\in [0,1]$. 

\noindent In the following, we show that the previous implication is also true for non-compact spaces with respect to the relation $\mathcal{C}^+$ as shown in Theorem \ref{th_conj}. Let us start by proving the following lemma.
\begin{lem}\label{lem_conj}
    Let $(X,f)$ and $(Y,g)$ be two dynamical systems with $(X,d_X)$ and $(Y,d_Y)$ metric spaces. If a homeomorphism $h:X\rightarrow Y$ is a topological conjugacy, then for $x,y\in X$,
    $$x\,\mathcal{C}^+\, y \iff h(x)\,\mathcal{C}^+\, h(y).$$
\end{lem}
\begin{proof}
The proof proceeds along the same lines of the proof of \cite[Theorem 2.7]{kumarlal21} with the modification that we consider two possibly distinct points $x,y\in X$ such that $x\,\mathcal{C}^+\, y$, instead of considering a single point $x\in X$ such that $x\,\mathcal{C}^+\, x$. 

\noindent Let $\varepsilon\in C(X,\mathbb{R}^+)$. Since $h$ is continuous, by \cite[Lemma 2.3]{kumarlal21}, there exists $\delta\in C(X,\mathbb{R}^+)$ such that $d_Y(h(x_1),h(x_2))<\varepsilon(h(x_1))$ whenever $d_X(x_1,x_2)<\delta(x_1)$. Therefore, for all $i=1,\ldots,n-1$, $$d_X(f(x_i),x_{i+1})<\delta(f(x_i)).$$ Set $y_i:=h(x_i)$ for all $i=1,\ldots,n$. Since $h\circ f=g\circ h$ and $d_X(f(x_i),x_{i+1})<\delta(f(x_i))$, it follows that $$d_Y(g(y_i),y_{i+1})=d_Y(h\circ f(x_i),h(x_{i+1}))<\varepsilon(h\circ f(x_i))=\varepsilon(g(y_i)).$$
Hence $h(x)=y_1,\ldots,y_n=h(y)$ form an $\varepsilon$-chain from $h(x)$ to $h(y)$. The converse implication follows from an analogous argument applied using $h^{-1}$ instead of $h$.
\end{proof}
\noindent The following theorem is a straightforward consequence of the above lemma.
\begin{thm}\label{th_conj}
Let $(X,f)$ and $(Y,g)$ be two topologically conjugate dynamical systems with $(X,d_X)$ and $(Y,d_Y)$ metric spaces. Then, the posets $(\mathfrak{C}^+_f,\preceq^+)$ and $(\mathfrak{C}^+_g,\preceq^+)$ are order isomorphic.
\end{thm}

\section*{Conclusion}

This work has explored the structural possibilities and limitations of chain component posets arising from discrete dynamical systems. We have shown that compactness and continuity impose tight constraints: certain order types, such as countable dense linear orders, cannot be realized under these assumptions. At the same time, we showed that a wide range of well-ordered posets of the form \(\lambda + 1\) (for \(\lambda < \omega_1\)) are realizable, highlighting the richness available even within the compact continuous setting. In contrast, removing regularity assumptions on the map allows for the construction of densely ordered structures that would otherwise be excluded.
The results illustrate the significant interplay between the topological properties of the phase space and structural features of the dynamics.



\bibliography{sn-bibliography}

@book{co78,
  author		    = "Conley, C. C.",
  title			= "Isolated invariant sets and the Morse index",
  volume          = "38",
  publisher		= "American Mathematical Soc.",
  year			= "1978",
address          ="Providence, RI",
  doi			    = "10.1090/cbms/038"
}

@book{kurka,
  author		    = "Kurka, P.",
  title			= "Topological and symbolic dynamics",
  address		    = "France",
  publisher		= "Soci{\'e}t{\'e} math{\'e}matique de France Paris",
  year			= "2003",
}

@book{ma15,
  author		    = "Manetti, M.",
  title			= "Topology",
  volume		    = "153",
  publisher		= "Springer Nature",
  year			= "2023",
  address         ="Switzerland",
  doi			    = "10.1007/978-3-031-32142-9"
}

@incollection{mischaikow,
author = "Konstantin, M. and Marian, M.",
title     = "Chapter 9 - Conley Index",
editor    = "Bernold Fiedler",
series    = "Handbook of Dynamical Systems",
publisher = "Elsevier Science",
volume    = "2",
pages = "393-460",
year = "2002",
booktitle = "Handbook of Dynamical Systems",
address = "Netherlands",
issn = "1874-575X",
doi = "10.1016/S1874-575X(02)80030-3"
}

@article{norton,
  title= "The Conley decomposition theorem for maps: a metric approach",
  author="Norton, D. E.",
  journal="Commentarii Mathematici Universitatis Sancti Pauli",
  volume="44",
  number="2",
  pages="151--173",
  year="1995"
}

@article{hu91,
  title="Chain recurrence and attraction in non-compact spaces",
  author="Hurley, Mike",
  journal="Ergodic Theory and Dynamical Systems",
  volume="11",
  number="4",
  pages="709--729",
  year="1991",
  publisher="Cambridge University Press",
doi="10.1017/S014338570000643X"
}

@article{hu92,
  title={Noncompact chain recurrence and attraction},
  author={Hurley, M.},
  journal={Proceedings of the American Mathematical Society},
  volume={115},
  number={4},
  pages={1139--1148},
  year={1992},
doi={10.1090/S0002-9939-1992-1098401-X}
}

@article{hu98,
  title={Lyapunov functions and attractors in arbitrary metric spaces},
  author={Hurley, M.},
  journal={Proceedings of the American mathematical society},
  volume={126},
  number={1},
  pages={245--256},
  year={1998},
doi={10.1090/S0002-9939-98-04500-6}
}

@article{steele17,
  title={Dynamics of typical Baire-1 functions on the interval},
  author={Steele, T. H.},
  journal={Journal of Applied Analysis},
  volume={23},
  number={2},
  pages={59--64},
  year={2017},
  publisher={De Gruyter},
doi={10.1515/jaa-2017-0009}
}

@article{steele18,
  title={The space of $\omega$-limit sets for Baire-1 functions on the interval},
  author={Steele, T. H.},
  journal={Topology and its Applications},
  volume={248},
  pages={59--63},
  year={2018},
  publisher={Elsevier},
doi={10.1016/j.topol.2018.08.008}
}

@article{steele19,
  title={Dynamics of Baire-1 functions on the interval},
  author={Steele, T. H.},
  journal={European Journal of Mathematics},
  volume={5},
  number={1},
  pages={138--149},
  year={2019},
  publisher={Springer},
doi={10.1007/s40879-018-0234-0}
}

@article{steelebaire2,
  title={Dynamics of Baire-2 functions on the interval},
  author={Steele, T. H.},
  journal={Topology and its Applications},
  volume={265},
  pages={106821},
  year={2019},
  publisher={Elsevier},
doi={10.1016/j.topol.2019.106821}
}

@article{crannellmartelli,
  title={Dynamics of quasicontinuous systems},
  author={Crannell, A. and Martelli, M.},
  journal={Journal of Difference Equations and Applications},
  volume={6},
  number={3},
  pages={351--361},
  year={2000},
  publisher={Taylor \& Francis},
doi={10.1080/10236190008808234}
}

@article{crannel05,
  title={Closed relations and equivalence classes of quasicontinuous functions.},
  author={Crannell, A. and Frantz, M. and LeMasurier, M.},
pages={409--423},
volume={31},
number={2},
journal={Real Anal. Exchange},
  year={2005}
}

@article{pawlakdarboux,
  title={Dynamics of Darboux functions},
  author={Pawlak, H. and Pawlak, R. J.},
  journal={Tatra Mt. Math. Publ},
  volume={42},
  pages={51--60},
  year={2009},
doi={10.2478/v10127-009-0005-x}
}

@inproceedings{pawentropy,
  title={On the entropy of Darboux functions},
  author={Pawlak, R. J.},
  booktitle={Colloq. Math},
  volume={116},
  number={2},
  pages={227--241},
  year={2009},
doi={10.4064/cm116-2-7}
}

@article{pawlak19,
  title={On almost continuous functions and peculiar points},
  author={Loranty, A. and Pawlak, R. J. and Terepeta, M.},
  journal={European Journal of Mathematics},
  volume={5},
  number={1},
  pages={106--115},
  year={2019},
  publisher={Springer},
doi={10.1007/s40879-018-0264-7}
}

@inproceedings{szuca10,
  title={On Pawlak’s problem concerning entropy of almost continuous functions},
  author={Natkaniec, T. and Szuca, P.},
  booktitle={Colloq. Math},
  volume={121},
  number={1},
  pages={107--111},
  year={2010},
doi={10.4064/cm121-1-9}
}

@article{ciklova,
  title={Dynamical Systems generated by functions with connected {G}$_\delta$ graphs},
  author={Ciklov{\'a}, M.},
  journal={Real Analysis Exchange},
  volume={30},
  number={2},
  pages={617--637},
  year={2005},
  publisher={Michigan State University Press}
}

@article{chainsregularity,
  title={Chains without regularity},
  author={Della Corte, Alessandro and Farotti, Marco},
  journal={Journal of Differential Equations},
  volume={464},
  pages={114222},
  year={2026},
  publisher={Elsevier},
  doi={https://doi.org/10.1016/j.jde.2026.114222}
}

@article{kumarlal21,
  title={Attractors and chain recurrence in noncompact space for semigroup of continuous maps},
  author={Kumar, S. and Lalwani, K.},
  journal={arXiv preprint arXiv:2111.08252},
  year={2021},
doi={10.48550/arXiv.2111.08252}
}

@article{alikhani2019chain,
  title={On chain recurrent sets of typical bounded Baire one functions},
  author={Alikhani-Koopaei, A.},
  journal={Topology and its Applications},
  volume={257},
  pages={1--10},
  year={2019},
  publisher={Elsevier},
doi= "10.1016/j.topol.2019.02.011"
}

@article{alikhani2022iterates,
  title={Iterates of Borel functions},
  author={Alikhani-Koopaei, A. and Steele, T. H.},
  journal={Topology and its Applications},
  volume={320},
  pages={108237},
  year={2022},
  publisher={Elsevier},
doi= "10.1016/j.topol.2022.108237"
}

@article{llibre2015birth,
  title={On the birth of limit cycles for non-smooth dynamical systems},
  author={Llibre, J. and Novaes, D. D. and Teixeira, M. A.},
  journal={Bulletin des sciences mathematiques},
  volume={139},
  number={3},
  pages={229--244},
  year={2015},
  publisher={Elsevier},
doi="10.1016/j.bulsci.2014.08.011"
}

@article{llibre2015maximum,
  title={Maximum number of limit cycles for certain piecewise linear dynamical systems},
  author={Llibre, Jaume and Novaes, Douglas D and Teixeira, Marco A},
  journal={Nonlinear Dynamics},
  volume={82},
  number={3},
  pages={1159--1175},
  year={2015},
  publisher={Springer},
doi="10.1007/s11071-015-2223-x"
}

\end{document}